\documentclass[a4paper,
               11pt,
               pdftex,
               normalheadings,
               headsepline,
               footsepline,
               onecolumn,
               headinclude,
               footinclude,
               DIV14,
               abstracton]
{scrartcl}

\usepackage
{
    graphicx,
    amssymb,
    amsmath,
    amsthm,
    xcolor,
    dsfont,
    algpseudocode,
    authblk,
}

\usepackage[ngerman, english]{babel}

\usepackage[
left=25mm,
right=25mm,
top=20mm,
bottom=35mm
]{geometry}

\usepackage{subcaption}

\usepackage[colorlinks,
            pdffitwindow=false,
            plainpages=false,
            pdfpagelabels=true,
            pdfpagemode=UseOutlines,
            pdfpagelayout=SinglePage,
            bookmarks=false,
            colorlinks=true,
            hyperfootnotes=false,
            linkcolor=blue,
            citecolor=green!50!black]
{hyperref}

\usepackage{bm}
\usepackage{cleveref}

\usepackage{enumitem}
\usepackage{algorithm}

\usepackage{tabularx}

\usepackage{todonotes}
\usepackage[normalem]{ulem}

\newcommand{\Khat}{\hat{K}}
\newcommand{\Ktilde}{\tilde{K}}

\newcommand{\R}{\mathbb{R}}

\newcommand{\Acal}{\mathcal{A}}

\newcommand{\Fcal}{\mathcal{F}}

\newcommand{\Hcal}{\mathcal{H}}
\newcommand{\Kcal}{\mathcal{K}}
\newcommand{\Kcalhat}{\hat{\Kcal}}

\newcommand{\Ncal}{\mathcal{N}}

\newcommand{\Ycal}{\mathcal{Y}}

\newcommand{\identity}{\operatorname{id}}

\newcommand{\groupaction}[2]{{#1}\left[{#2}\right]}
\newcommand{\invgroupaction}[2]{{#1}^{-1}\left[{#2}\right]}

\newcommand{\norm}[1]{\left\lVert#1\right\rVert}
\newcommand{\abs}[1]{\left\vert#1\right\vert}

\newcommand{\dint}[1]{\,\mathrm{d}#1}
\newcommand{\parder}[2]{\frac{\partial #1}{\partial #2}}
\newcommand{\parderHigher}[3]{\frac{\partial^{#3} #1}{\partial #2^{#3}}}

\newtheorem{theorem}{Theorem}[section]
\newtheorem{definition}{Definition}
\newtheorem{remark}{Remark}
\newtheorem{example}{Example}[section]

\let\OLDthebibliography\thebibliography
\renewcommand\thebibliography[1]{
	\OLDthebibliography{#1}
	\setlength{\parskip}{0pt}
	\setlength{\itemsep}{0pt plus 0.3ex}
}

\date{}

\begin{document}

\title{Equivariance and partial observations in Koopman operator theory for partial differential equations}
\author[1,2]{\normalsize Sebastian Peitz}
\author[3]{Hans Harder}
\author[4]{Feliks Nüske}
\author[5]{Friedrich Philipp}
\author[6]{Manuel Schaller}
\author[5]{Karl Worthmann}
\affil[1]{\normalsize Department of Computer Science, TU Dortmund, Dortmund, Germany}
\affil[2]{\normalsize Lamarr Institute for Machine Learning and Artificial Intelligence, Dortmund, Germany}
\affil[3]{\normalsize Department of Computer Science, Paderborn University, Paderborn, Germany}
\affil[4]{\normalsize Max Planck Institute for Dynamics of Complex Technical Systems, Magdeburg, Germany}
\affil[5]{\normalsize Institute of Mathematics, TU Ilmenau, Ilmenau, Germany}
\affil[6]{\normalsize Institute of Mathematics, TU Chemnitz, Chemnitz, Germany}

\maketitle

\begin{abstract}
The Koopman operator has become an essential tool for data-driven analysis, prediction and control of complex systems. The main reason is the enormous potential of identifying linear function space representations of nonlinear dynamics from measurements. 
This equally applies to ordinary, stochastic, and partial differential equations (PDEs). Until now, with a few exceptions only, the PDE case is mostly treated rather superficially, and the specific structure of the underlying dynamics is largely ignored.
In this paper, we show that symmetries in the system dynamics can be carried over to the Koopman operator, which allows us to significantly increase the model efficacy. Moreover, the situation where we only have access to partial observations---i.e., measurements, as is very common for experimental data---has not been treated to its full extent, either. Moreover, we address the highly-relevant case where we cannot measure the full state, where alternative approaches (e.g., delay coordinates) have to be considered. We derive rigorous statements on the required number of observables in this situation, based on embedding theory. We present numerical evidence using various numerical examples including the wave equation and the Kuramoto-Sivashinsky equation.
\end{abstract}

\section{Introduction}
Many phenomena in nature can be described by \emph{partial differential equations (PDEs)}, where the system state depends on both space and time. Popular examples are fluid dynamics or electromagnetics. Studying such systems poses many challenges, including (the derivation of) sophisticated numerical discretization schemes (using, e.g., finite elements), a very high dimension of the resulting discretized nonlinear systems, and the challenge that in real experiments, the entire system state is accessible in few cases only. In addition, we only have very crude models---if at all---for some systems, e.g., from biology. 

Due to these reasons, there has been an increasing interest in the scientific community to develop and improve methods to infer and predict dynamical systems from data. Popular examples are the \emph{Sparse Identification of Nonlinear Dynamics} \cite{BPK16}, statistical approaches such as the \emph{Mori-Zwanzig} framework \cite{CHK00}, or techniques based on deep learning \cite{VBW+18}, also using physical knowledge \cite{RPK19}. Another approach that has been particularly successful in the past decade is the \emph{Koopman operator} framework \cite{Koo31,RMB+09,Mez13}. The driving force behind this renaissance is twofold: (1) The Koopman operator yields a linear representation of nonlinear dynamical systems, thus giving us access to powerful techniques from linear systems theory, and (2) the advances in numerical approximation (here in the form of the \emph{Extended Dynamic Mode Decomposition (EDMD)} \cite{WKR15,KKS16,KNP+20}) now allow us to identify the Koopman operator from measurements using simple linear regression.
As a consequence, the Koopman operator has been studied extensively for data-driven analysis \cite{BBKK22}, coarse graining \cite{BNC18,KNP+20,NKS21,NKBC21} and control \cite{PBK15,KM18,PK19,POR20,SWP+23,BGSW23} of a very large number of applications including neuroscience, social systems, robotics, and fluid dynamics. Recent results on finite-data error bounds can be found in \cite{ZZ22,NPP+23,PSW+24,BML+23,KPS+24}. 

However, PDEs have until now mostly been treated in a rather ad-hoc manner, meaning that the observable consists of the state values at the grid nodes of some numerical discretization scheme \cite{RMB+09,Sch10}. Alternatively, kernel-based observable functions have been used---again with discretized full-state observables---to avoid dealing with the infinite-dimensional nature of PDEs \cite{WRK15,KPB16}. A machine-learning-based approach was presented in \cite{OR19}, where an autoencoder was trained to map the full-state of the PDE---using the values on the grid nodes---to a latent space of low dimension. In \cite{OPR24}, a method based on a Kalman filter was developed to avoid defining the observable function altogether.
When it comes to a more formal treatment of PDEs, the literature is much scarcer. In \cite{PK18}, for instance, the Koopman operator of the Burgers equation was derived analytically, albeit with the help of first linearizing the dynamics using the Cole-Hopf transformation.
The most formal treatment of Koopman operator methods for PDEs can be found in \cite{NM20,Mau21}.

Despite its recent increase in popularity, there is until now only little literature on Koopman operators for systems with symmetries; see, e.g., \cite{SER+19,sinha2020koopman} for ordinary differential equations, or \cite{WSGK22}, where the matrix approximation of the Koopman operator was used to heuristically identify symmetries in Markov Decision Processes. The central goal of this paper is thus to provide a systematic treatment of symmetries in PDEs (Section \ref{sec:Equivariance}), which will allow us to significantly increase the numerical performance of the Koopman operator approximation. 

A second aspect that has until now received a mostly heuristic treatment, is the question of partial observations and unknown state spaces. Even though the operator can be approximated for many different types of observation functions, many articles simply use the \emph{full state observable} $f(y)=y$, which greatly simplifies the situation. 
There are---of course---several exceptions. For instance, an attempt in this direction was proposed in \cite{CKH22} using nonlinear input transformations, or in \cite{OPR24}, where a Kalman filter was used to infer the state of an unknown system from measurements. Finally, delay coordinates have been successfully used in the so-called \emph{Hankel DMD} \cite{AM17}, also in various applications \cite{BBP+17,AM20,KKBK20}.
However, there are severe pitfalls when we do not know or cannot efficiently discretize the system's state space. 
Most importantly, it is not clear how many measurements are required to obtain a reliable approximation of the Koopman operator associated with the system dynamics.
Here, we will address the issue of unknown state spaces and/or partial measurements in detail and derive rigorous conditions on partial observations (Section \ref{sec:Partial}) that rely on fundamental embedding theorems. 
It should be noted that a similar approach was considered in \cite{BKK+17} for stochastic dynamical systems in the context of molecular dynamics.

Finally, we show in Section \ref{sec:EqivPlusPartial} that a combination of these results allows us to design efficient schemes for EDMD, and how to transfer a learned Koopman approximation from one domain to another without retraining.

\section{Koopman operator for PDEs}
We start by introducing the Koopman operator for partial differential equations, although everything that follows applies equally to ordinary differential equations (ODEs). Consider the general dynamical system
\begin{equation}\label{eq:PDE}
    \parder{y}{t} = \Ncal(y), \quad y(\cdot,0) = y_0.
\end{equation}
Here, $y(x,t)$ is the space- and time-dependent system state, $t \in \R_{\geq 0}$ is the time and $x\in\Omega$ the spatial coordinate, where $\Omega \subset \R^n$ is the spatial domain with boundary $\partial \Omega$. For simplicity of notation, we will restrict our analysis to one-dimensional spatial domains. However, extensions to higher dimensions are possible in a straight-forward manner. Moreover, the state at time $t$ is an element of a function space with appropriate differentiability (e.g., a Sobolev space $y(\cdot, t)\in H^p(\Omega)$) and $\Ncal: D(\Ncal) \rightarrow H^p(\Omega)$ (with $D(\Ncal)\subset H^p(\Omega)$ being the domain of $\Ncal$) is a nonlinear partial differential operator describing the dynamics of the system. We will for now assume periodic boundary conditions (BCs). The system's flow map $\Phi^\tau : L^2(\Omega)\to L^2(\Omega)$ is defined by
\[
    y(\cdot,t+\tau) = \Phi^\tau(y(\cdot,t)).
\]
We assume that the solution to \eqref{eq:PDE} exists and that it is unique. Furthermore, we assume that the system possesses an invariant compact set $\Acal\subset D(\mathcal N)$, i.e., $\Phi^\tau(\Acal) = \Acal$, which has dimension~$\dim(\Acal)=d$ (often, the \emph{box counting dimension} is used in this context \cite{ZDG19}). 
Usually, $\Acal$ is the system's \emph{attractor} or an \emph{invariant manifold}, and it is well known that many PDEs possess an attractor with $d<\infty$.
\begin{example}\label{ex:KS}
    The well-known \emph{Kuramoto--Sivashinsky} equation that we will study in this paper is, in dimensionless form, given by
\begin{equation}\label{eq:KS}
    \parder{y}{t} + 4\parderHigher{y}{x}{4} + \mu \left[ \parderHigher{y}{x}{2} + y \parder{y}{x} \right] = 0
\end{equation}
    on the domain $\Omega=(0,L)=(0,2\pi)$ with $\mu \in (0,\infty)$. Depending on $\mu$, the system exhibits rich dynamics, from bimodal fixed points to traveling waves to fully chaotic behavior \cite{HNZ86}. Moreover, one can bound the dimension of the attractor $\Acal$ in terms of the domain size $L$ via $d\leq L^{2.46}$ when not considering the non-dimensionalized version as we do in \eqref{eq:KS}, see \cite{ZDG19} for more details.
\end{example}

The \emph{semigroup of Koopman operators} associated with \eqref{eq:PDE} is defined on a space of observable functionals, see~\cite{Mau21}.

\begin{definition}[Koopman semigroup and generator]\label{def:Koopman}
Consider the space $C(\Acal)$ of continuous real-valued functionals $f: \Acal \to \R$, endowed with the supremum norm $\norm{f} = \sup_{y\in\Acal}\abs{f(y)}$.
The \emph{semigroup of Koopman  operators} $(\Kcal^\tau)_{\tau\geq 0}$ associated with the semiflow $(\Phi^\tau)_{\tau\geq 0}$ is defined by
\begin{equation}\label{eq:Koopman}
\Kcal^\tau f = f \circ \Phi^\tau, \qquad f\in C(\Acal).
\end{equation}
The Lie generator of the semigroup is the linear operator $\Kcal: D(\Kcal) \rightarrow C(\Acal)$ that satisfies
\begin{align}\label{eq:Koopman_generator}
    \left(\Kcal f\right)(y) = \lim_{\tau\to 0} \frac{\left(\Kcal^\tau f\right)(y) - f(y)}{\tau}\qquad \forall y\in\Acal.
\end{align}
\end{definition}

\begin{remark}
    The Lie generator of the Koopman semigroup has a close connection to generators of strongly continuous semigroups, with the distinction that the limit is defined in the strong sense in the latter case, see \cite[Remark 1]{Mau21} for a more detailed discussion and additional references.
\end{remark}

\begin{remark}\label{rem:vector}
An extension to vector-valued observable functions $f: \Acal \rightarrow \R^q$ in $C(\Acal)^q$ can be realized in a straightforward manner, see, e.g., \cite{BMM12}.
\end{remark}

In the case of a Koopman semigroup associated with a semiflow generated by the PDE \eqref{eq:PDE}, it follows from the chain rule that the generator is given by
\begin{equation}\label{eq:Koopman_DE}
    (\Kcal f)(y) = G_{\Ncal(y)}f(y),
\end{equation}
which can be interpreted as the Lie derivative associated with the infinite-dimensional vector field $\Ncal$, where $G_{\Ncal(y)}f(y)$ denotes the (linear) Gâteaux derivative of $f$ at $y$ in the direction $\Ncal(y)$. 
Note that this is closely related to the generator PDE that we find for ODEs \cite{KNP+20}. An alternative derivation using the functional derivative can be found in \cite{NM20}.

Even though the Koopman operator formalism has been known for a very long time, it has received a massive increase in attention in the past decade, mostly due to the advances in its numerical approximation via EDMD. Based on the observation that the operators in Eqs.\ \eqref{eq:Koopman} and \eqref{eq:Koopman_generator} are linear, we can try to compute finite-dimensional approximations $K^\tau$ and $K$ of $\Kcal^\tau$ and $\Kcal$, respectively. We achieve this via Galerkin projection by introducing a finite dictionary $\{\psi_j\}_{j=1}^{\ell}$ of functions $\psi_j\in C(\Acal)$,
and coefficients $a\in\R^\ell$: 
\begin{equation}\label{eq:Galerkin}
	f(y(\cdot,t)) \approx \sum_{j=1}^\ell a_j \psi_j(y(\cdot,t)) = a^\top \Psi(y(\cdot,t)).
\end{equation}
The goal is then to approximate the compression $P_{\mathbb{V}} \mathcal{K}^\tau|_{\mathbb{V}}$, where $\mathbb{V} = \operatorname{span}\{ \psi_j| j \in \{1,\ldots,\ell\} \}$, and $P_{\mathbb{V}}$ is the projection onto $\mathbb{V}$.
Using time series data $\{\Psi(y_k)\}_{k=0}^m$,  
where $y_i=y(\cdot,i\tau)$ and $\Psi = [\psi_1^\top,\cdots,\psi_\ell^\top]^\top$, we can now simply use linear regression to find the best fit matrix $K^\tau$:
\begin{equation}\label{eq:Koopman_regression}
\min_{K^\tau\in\R^{\ell \times \ell}} \sum_{i=0}^{m-1} \norm{\Psi(y_{i+1}) - K^\tau \Psi(y_i)}_2^2.
\end{equation}
A very similar regression problem can be formulated to approximate the generator $\Kcal$ via $K$ \cite{KNP+20}.
In both cases, it can be shown that this matrix converges to the Galerkin projection of $\Kcal^\tau$ in the infinite data limit $m\rightarrow\infty$ \cite{WKR15,KKS16}, and to the true Koopman operator when additionally $\ell\rightarrow \infty$ \cite{KM18b}. 
Moreover, finite-data error bounds can be found in \cite{Mez22,ZZ22,NPP+23,PSW+24,KPS+24}, using either i.i.d.\ or ergodic sampling. 
Most of them have until now considered ordinary or stochastic differential equations. A more general setting can be found in \cite{PSB+24}, which covers the PDE case as well.

\section{Equivariant Koopman operators for equivariant PDEs}
\label{sec:Equivariance}
It is a well-known fact that a large number of PDE systems exhibits symmetries, i.e., the dynamics are equivariant under certain group actions (to be specified shortly), which depend on the system dynamics $\Ncal$, as well as on the domain $\Omega$ and the boundary conditions on $\partial \Omega$. For instance, an alternative version of Eq.\ \eqref{eq:KS} from Example \ref{ex:KS} in coordinate-free form is
\begin{equation}\label{eq:KS_alternative}
    \parder{y}{t} = \underbrace{- 4\Delta^2 y - \mu \left[ \Delta y + \frac{1}{2} \left| \nabla y \right|^2 \right]}_{=\Ncal(y)},
\end{equation}
from which one can derive \eqref{eq:KS} by taking the derivative with respect to $x$ and substituting $\partial y / \partial x$ by $y$.
In Eq.\ \eqref{eq:KS_alternative}, we quickly see that the partial differential operator $\Ncal$ is equivariant under translations (for any spatial dimension) as well as rotations (in 2D or 3D). This means that shifting (or rotating) the solution $y$ and then applying $\Ncal$ yields the same result as shifting (or rotating) $\Ncal(y)$. The goal of this section is to show that the Koopman operator inherits such symmetries.

\subsection{Some prerequisites on equivariant systems}
In order to introduce the basic symmetry concepts, let us disregard the time dependence of $y$ for a moment and instead denote the state at a fixed time $t$ by $y_t(x)=y(x,t)$ for $x\in\Omega$ 
For a more detailed introduction, see \cite{BBCV21}. A \emph{group} is a set $G$ along with an associative composition operation $\circ : G \times G \rightarrow G, (g, h) \mapsto g h$ that contains an identity and inverses. 
A \emph{group action} of $G$ on a set $\Omega$ is then defined as a mapping $(g,x) \rightarrow \groupaction{g}{x}$ associating a group element $g\in G$ and a point $x\in\Omega$ with some other point in $\Omega$ in a way that is compatible with the group operations, i.e., $\groupaction{g}{(\groupaction{h}{x})}=\groupaction{(gh)}{x}$ for all $g, h \in G$ and $x \in \Omega$.
\begin{example}
    The Euclidean group $E(2)$ in the plane is the group of transformations of $\R^2$ that preserves Euclidean distances, i.e., it consists of translations, rotations, and reflections. The same group can also act on the space of functions on the plane, that is, if we have a group $G$ acting on $\Omega$, we automatically obtain an action of $G$ on the space $\Ycal(\Omega)$: $\groupaction{g}{y_t}(x) = y_t(\invgroupaction{g}{x})$.
    Due to the inverse on $g$, this is indeed a valid group action, in that we have 
    $(\groupaction{g}{(\groupaction{h}{y_t})})(x)=(\groupaction{(gh)}{y_t})(x)$.
\end{example}
If $G$ acts on two sets $X$ and $Y$ and $f : X \rightarrow Y$ is some function, we say that $f$ is \emph{equivariant} if it satisfies $\groupaction{g}{f(x)} = f(\groupaction{g}{x})$, and we say that it is \emph{invariant} if it satisfies $f(x) = f(\groupaction{g}{x})$. For example, the flow $\Phi^\tau : \Ycal \rightarrow \Ycal$ is often times equivariant with respect to a given group action on $\Ycal$.

\subsection{Related symmetry concepts in machine learning}
Before addressing the particular case of Koopman operators in the following subsection, we would like to highlight that the concept of exploiting symmetries has recently gained a lot interest in a large number of sub-fields of machine learning research.
Starting with convolutional neural networks---which respect translational symmetries by design---many of these approaches can today be subsumed under the umbrella term \emph{geometric deep learning} \cite{BBCV21}. The unifying objective is to construct learning algorithms that respect known symmetries by design, which often leads to a significant increase in performance and a reduction in the required training data; see \cite{MBB+15,CW16,BBL+17} for a few early references and \cite{AGS21,BBCV21} for comprehensive overviews. Further examples include reinforcement learning \cite{PWH+20,VRV+23,PSC+24} (also in combination with Koopman operators \cite{WSGK22}), reservoir computing \cite{PHG+18} or extreme learning machines \cite{HRV+24}.

\subsection{Equivariant Koopman operators}
In order to formalize our discussion, we will consider observables $f$ that can be defined via convolutions. As a special case, this yields point measurements, but it covers many other settings as well.
We thus define a corresponding \emph{convolution operator}.
Generally, a convolution is the composition of two functions which produces another function in a new coordinate. Usually, a function of interest (e.g., our system state $y$) is composed with a kernel $\theta$:
\begin{equation*}
    (y_t \star \theta)(s)= \int_{\Omega} y_t(x) \theta(s-x) \dint{x}.
\end{equation*}
In many situations, $\theta$ is a Gaussian kernel---i.e., $\theta=\exp(-\|s-x\|^2 / \alpha)$, with kernel bandwidth $\alpha\in\R^{>0}$---that somewhat ``localizes'' $y_t$ around the center $s$ (even though not in a strict sense, of course).
Now, fixing $s$, we can define the observable function
\begin{equation}\label{eq:observable}
    f_s(y_t) = (y_t \star \theta)(s).
\end{equation}

\begin{remark}\label{rem:conv}
    Eq.\ \eqref{eq:observable} is very general, as it includes observables modeling point evaluations (using a Dirac-Delta kernel) as well as integrals over parts of the domain. If we instead want to consider nonlinear functions of the state, we can still remain in the convolution setting by first applying a nonlinear function---e.g., in a point-wise manner---and then the convolution operation.
\end{remark}

Let us now assume that our group action is a shift operation, $\groupaction{g}{x} = x-g~ \mathsf{mod}~L$, where $G$ is the group $[0,L)$ equipped with modular addition.
Furthermore, we assume that the partial differential operator $\Ncal$ defined by Eq.\ \eqref{eq:PDE} does not explicitly depend on space $x$ or time $t$.
We thus obtain shift equivariance of the right-hand side of the PDE and thus, the group action commutes with the flow $\Phi^t$ \cite[Remark 1]{OP21}:
\begin{equation*}
    \groupaction{g}{\Phi^\tau(y_t)} = \Phi^\tau(\groupaction{g}{y_t}),
\end{equation*}
where $\groupaction{g}{y_t}(x) = y_t(\invgroupaction{g}{x})$. 
Using the convolution observable and the equivariance of the PDE, we find that the corresponding Koopman operator inherits the equivariance property.
\begin{theorem}
    Consider a PDE of the general form \eqref{eq:PDE} with periodic boundary conditions, where $\Ncal$ does not explicitly depend on space $x$ or time $t$ and is thus equivariant under translations in $x$ (and $t$) in view of the periodic boundary conditions. Further assume that the observable $f_s$ (of the form as defined in \eqref{eq:observable}) is periodic as well.
    Then, the Koopman operator associated with \eqref{eq:PDE} is also equivariant under the same group action.
\end{theorem}
\begin{proof}
Introducing $\tilde{x} = x-g$, we get (under periodic BCs)
\begin{equation*}
\begin{aligned}
    f_{\invgroupaction{g}{s}}(y_t) &=  \int_{\Omega} y_t(x) \theta(s+g-x) \dint{x} \\
    &= \int_{\Omega} y_t(\tilde{x}+g) \theta(s-\tilde{x}) \dint{\tilde x}  \\
    &= \int_{\Omega} y_t(\invgroupaction{g}{\tilde{x}}) \theta(s-\tilde{x}) \dint{\tilde x} \\
    &= (\groupaction{g}{y_t} \star \theta)(s)= \groupaction{g}{(y_t \star \theta)}(s) \\
    &=\groupaction{g}{f_{(\cdot)}(y_t)}(s),
\end{aligned}
\end{equation*}
where we have exploited the linearity of both the convolution operation and the group action in line four in order to exchange the operations.
Moreover, we assume that $\theta$ is periodic as well.
As a consequence, we also obtain equivariance of the action of the Koopman operator. For any $y_0\in\Acal$ with $y_1=\Phi^\tau\in\Acal$, we find
\begin{equation*}
    \begin{aligned}
    \left(\Kcal^t f_{\invgroupaction{g}{s}}\right)(y_0) &= f_{\invgroupaction{g}{s}} \left(\Phi^t(y_0)\right) \\
    &= \groupaction{g}{f_{(\cdot)}\left(\Phi^t(y_0)\right)}(s) \\
    &= \groupaction{g}{\left(\Kcal^t f_{(\cdot)}\right)(y_0)}(s).
    \end{aligned}
\end{equation*}
\end{proof}
What follows is that the same Koopman operator approximation can be applied to different observables $f_{s_1}$ and $f_{s_2}$, where $s_2-s_1 = g \in \R$. We can thus compute a Koopman operator for some $f_{s}$, and apply the same operator of a shifted version of $f_{s}$.
This way, we obtain the possibility to decouple the Koopman operator from a specific spatial domain $\Omega$. As long as we remain within domains with periodic boundary conditions, a transfer is possible in a simple and straightforward manner. Moreover, the equivariance can easily be extended to vector-valued convolution observables, cf.\ Remark \ref{rem:vector}.

Note that the specific form of the observable \eqref{eq:observable} excludes global transformations such as the Fourier transform. For these, the symmetry transformation has more sophisticated implications that are less intuitive, which we emphasize in the following example. For a more thorough discussion of the requirements for symmetry-preserving observables, see \cite{HPN+24}.
\begin{example}[Fourier observable]\label{ex:FourierObservable}
    Let us assume that---instead of Eq.\ \eqref{eq:observable}---our observable $f$ is defined as the (complex) Fourier transform for some fixed frequency $\omega$, i.e.,
    \[
        f(y) = \Fcal\{y\}(\omega) = \int_{\R^n} y(x) e^{-i\omega x} \dint{x},
    \]
    and that our group action $g$ introduces a translation: $\groupaction{g}{y}(x)=y(x-g)$.
    We then find
    \begin{align*}
        f(\groupaction{g}{y}) &= \Fcal\{\groupaction{g}{y}\}(\omega) = \int_{\R^n} \groupaction{g}{y}(x) e^{-i\omega x} \dint{x} = \int_{\R^n} y(x-g) e^{-i\omega x} \dint{x} \\
        &= \int_{\R^n} y(x) e^{-i\omega (x+g)} \dint{x}
        = e^{-i\omega g} \int_{\R^n} y(x) e^{-i\omega x} \dint{x} \\
        &= e^{-i\omega g} f(y) = \groupaction{g}{f}(y).
    \end{align*}
    We thus see that the group action applied to $f$ results in the multiplication with a complex number.
\end{example}

\section{Partial measurements \& unknown state spaces}
\label{sec:Partial}
A large part of the existing literature focuses on small-scale systems or the situation where $f$ is the identity mapping, i.e., the \emph{full state observable} \cite{WKR15}. This means that the Koopman operator provides a linear system which advances the specific function (or functional) $f=\operatorname{Id}$ forward in time. Alternatively, it is at least assumed that the state space (here: $L^2(\Omega)$) is known and that the observable $f$ can be numerically approximated in an efficient manner---exceptions being approaches based on Hankel DMD or Kalman filtering, as discussed in the introduction.
However, this viewpoint has severe limitations. If, for instance, we consider PDEs, the state space may be challenging to approximate numerically. Moreover, if the data stems from real experiments, the domain may be unknown altogether. 
Finally---and this is central to efficient numerical approximations of equivariant Koopman operators---if we want to calculate local, equivariant EDMD models, then wee need to ensure that these local measurements yield an approximation that is consistent with the true dynamics.

\subsection{A common pitfall of partial measurements}
In the following, we will precisely consider the above-described situation where we do not necessarily know $f$ or the state space, not to mention the attractor $\Acal$. On an abstract level, there is an observable $f:\Acal \rightarrow \R^q$, according to which we collect our measurements $z_i = f(y_i)$. However, as we are ignorant of $f$ or the domain $\Acal$, we appear to be at an impasse: we cannot define a dictionary $\{\psi_j\}_{j=1}^{\ell}$ in $C(\Acal)^q$, which is the key step in EDMD, cf.\ Eqs.\ \eqref{eq:Galerkin}--\eqref{eq:Koopman_regression}.

A practical means to overcome this impasse is to collect measurements $z=[z_0,z_1,\ldots,z_m]\in\R^{q \times (m+1)}$, where $z_i = f(y_i)$, and try to approximate the Koopman operator directly from the data. 
If the measurement is low-dimensional (i.e., $q$ is small) one often simply \emph{lifts} the data $z$ using a dictionary such as delay coordinates or polynomials with maximal degree $s$, i.e.,
\begin{equation}\label{eq:EDMD_CDS}
    \Psi(z) = \begin{bmatrix} 1 & z_1 & \ldots & z_q & z_1^2 & z_1 z_2 & \ldots & z_q^{s}\end{bmatrix}^\top_{~\cdot}
\end{equation}
Examples of this approach are, e.g., lift and drag measurements of a fluid flow \cite{PK19}, coarse-grained coordinates of large molecules \cite{NKBC21} or delay coordinates of highway traffic data \cite{AM20}.
%
However, this approach provides a major pitfall when it comes to learning anything about the dynamics of the original system. 
Conceptually, we treat our measurements in such a way that we now use the full state observable on a different, implicitly defined dynamical system $\varphi^\tau: \R^q \rightarrow \R^q$ for the dynamics of the observed quantity $z$ on the artificial \emph{state space} $\R^q$:
\begin{equation}\label{eq:CDS}
    z_{i+1} = \varphi^\tau(z_i).
\end{equation}
Following \cite{DHZ16,ZDG19}, we will call \eqref{eq:CDS} the \emph{Core Dynamical System (CDS)}.

Assuming that the CDS exists and is uniquely defined, we can now define a new observable $h\in \Hcal= C(f(\Acal))^q$, $h: \R^q \rightarrow \R^q$ 
whose domain is now the state space of the CDS. 
In this setting, we can simply apply EDMD in its standard form (Eqs.\ \eqref{eq:Galerkin}--\eqref{eq:Koopman_regression} in combination with a dictionary as in \eqref{eq:EDMD_CDS}), to identify the Koopman operator associated with the CDS and the observable function $h$.\footnote{For convenience, we will use the full state observable $h=\identity$ here, but our equivalence result in Theorem \ref{thm:Koopman_CDS} also covers the more general setting for arbitrary $h$.} This concept is illustrated in Fig.\ \ref{fig:Koopman_CDS}, where $z=h(z) = h(f(y))$, and $\{\psi_j\}_{j=1}^{\ell}$ spans a subspace of $\Hcal$ instead of $C(\Acal)$.


\begin{figure}[h]
	\centering
    \includegraphics[width=.7\textwidth]{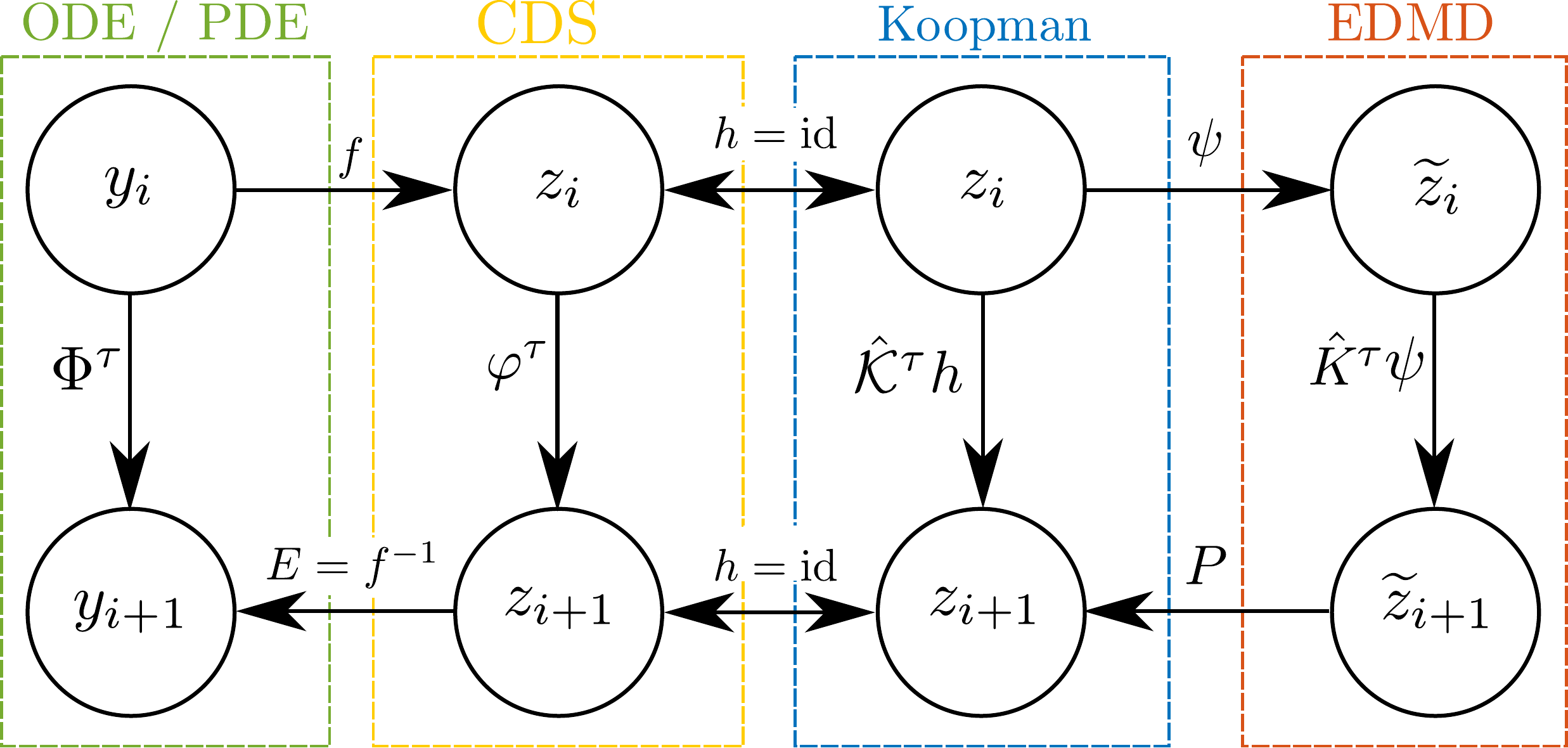}
	\caption{The extended Koopman operator concept for partially observed or unknown states. Instead of directly learning the Koopman operator for the observable $f:\Ycal \rightarrow \R^q$, we introduce the core dynamical system $\varphi^\tau$ as an intermediate model that---given a sufficiently large embedding dimension $q$---has a one-to-one correspondence to $\Phi^\tau$ on the attractor. The Koopman operator is then defined in the standard ODE setting using a new observable function $h\in\Hcal$, $h:\R^q \rightarrow \R^q$. For simplicity, we choose $h=\identity$ here.}
	\label{fig:Koopman_CDS}
\end{figure}

\subsection{Relation between the Core Dynamical System and the underlying PDE}
What remains to be shown is the correspondence between the original dynamical system $\Phi^\tau$ and the corresponding CDS $\varphi^\tau$. To this end, we closely follow the approach in \cite{ZDG19} and make use of well-known embedding theorems such as \cite{Whi36,Tak81,SYC91,Rob05}. 
For a detailed discussion on the more intricate implications of the following theorem---as well as further references regarding the terms \emph{prevalence}, \emph{box counting dimension} and \emph{thickness exponent} (which is $\sigma=0$ for many PDEs)---we refer the reader to \cite{ZDG19}.

\begin{definition}
    [\cite{ZDG19}, Definition 2.1]
    ~
    \begin{enumerate}[label=(\alph*)]
        \item A Borel subset $S$ of a normed linear space $V$ is \emph{prevalent} if there is a finite dimensional subspace $E$ of $V$ (the ``probe space'') such that for each $v\in V$, $v+e$ belongs to $S$ for (Lebesgue) almost every $e\in E$. Following a remark made in \cite{SYC91}, we will say that almost every map in a function space $V$ satisfies a certain property if the set of such maps is prevalent, even in the infinite dimensional case. Then this property will be called \emph{generic} (in the sense of prevalence).
        \item Let $Y$ be a Banach space, and let $\Acal \subset Y$ be compact. For $\varepsilon > 0$, denote by $N_Y(\Acal, \varepsilon)$ the minimal number of balls of radius $\varepsilon$ (in the norm of $Y$) necessary to cover the set $\Acal$. Then 
        \[
            \dim(\Acal) = \limsup_{\varepsilon \rightarrow 0} \frac{\log N_Y(\Acal, \varepsilon )}{-\log\varepsilon} = \limsup_{\varepsilon \rightarrow 0} -\log_{\varepsilon} N_Y(\Acal, \varepsilon)
        \]
        denotes the \emph{upper box counting dimension} of $\Acal$. 
        \item Let $Y$ be  a  Banach space,  and let $\Acal \subset Y$ be  compact.  For $\varepsilon > 0$,  denote  by $d_Y(\Acal, \varepsilon)$ the minimal dimension of all finite dimensional subspaces $V\subset Y$ such that every point of $\Acal$ lies within distance $\varepsilon$ of $V$. If no such $V$ exists, $d_Y(\Acal, \varepsilon) = \infty$. Then 
        \[
            \sigma(\Acal ,Y) = \limsup_{\varepsilon \rightarrow 0} - \log_{\varepsilon} d_Y(\Acal, \varepsilon )
        \]
        is called the \emph{thickness exponent} of $\Acal$ in $Y$.
    \end{enumerate}
\end{definition}

\begin{theorem}[\cite{Rob05,ZDG19}]\label{thm:embedding} 
Let $\Acal \subset \Ycal$ be a compact and invariant set with upper \emph{box counting dimension} $\dim(\Acal)=d$ and \emph{thickness exponent} $\sigma$. Choose an integer $q > 2(1 +\sigma )d$, and suppose further that the set $\Acal_p$ of $p$-periodic points of $\Phi^\tau$ satisfies $\dim(\Acal_p) < p/(2 + 2\sigma )$ for $p= 1,...,q$. 
Then for almost every (in the sense of \emph{prevalence}) Lipschitz map $f^d:\Ycal\rightarrow \R$ the \emph{delay observation map} $f:=D_q[f^d,\Phi^\tau]: \Ycal \rightarrow \R^q$ defined by 
\begin{equation*}
		y\mapsto\begin{bmatrix}
        f^d(y) & f^d(\Phi^\tau(y)) & \ldots & f^d(\Phi^{(q - 1)\tau}(y)
    \end{bmatrix}^\top
\end{equation*}
is one-to-one on $\Acal$. The same holds for a set of $q$ distinct observables $f^d_1,\ldots,f^d_q : \Ycal\to\R$, i.e.,
\begin{equation*}
    f = \begin{bmatrix}f^d_1(y) & \ldots & f^d_q(y)\end{bmatrix}^\top.
\end{equation*}
\end{theorem}

\begin{remark}
The \emph{\textbf{central message}} of the above theorem is that we can draw a close connection between the observable $f$ 
in the Koopman setting (Definition \ref{def:Koopman} and Remark \ref{rem:vector}) and the observation map $f^d$ in the embedding framework. The key statement for our purposes is that if the system dynamics of $\Phi^\tau$ is restricted to a compact set $\Acal$ with a finite dimension $d$, then we need to have at least $q>2d$ distinct measurements $f^d$---which jointly form the Koopman observable $f$---to obtain a one-to-one correspondence between $\Phi^\tau$ and $\varphi^\tau$:
\begin{equation*}
\Phi^\tau = E \circ \varphi^\tau \circ D_q[f^d,\Phi^\tau] = E \circ \varphi^\tau \circ f,
\end{equation*}
where $E$ is the inverse of $f$. This way, the CDS becomes:
\begin{equation}\label{eq:CDS_PDE}
    \varphi^\tau = f\circ\Phi^\tau\circ f^{-1}.
\end{equation}
\end{remark}

As a consequence, we can relate the Koopman operator for $\varphi^\tau$ to the Koopman operator for $\Phi^\tau$.

\begin{theorem}\label{thm:Koopman_CDS}
Let the assumptions of Theorem \ref{thm:embedding} hold and define the Koopman operator $\Kcal^\tau$ for the PDE \eqref{eq:PDE} in its standard form as in \eqref{eq:Koopman}. Furthermore, define the Koopman operator for the CDS $\varphi^\tau$ with observable $h:\R^q \rightarrow \R^p$ as follows:
\[
    \Kcalhat^\tau h = h \circ \varphi^\tau, \qquad h\in\Hcal.
\]
Then $h\circ \Kcal^\tau f = \Kcalhat^\tau h \circ f$. 
Moreover, we find that $\Kcal^\tau$ and $\Kcalhat^\tau$ share the same spectrum, and the eigenfunctions are related via $f$. 
\end{theorem}
\begin{proof}
The proof immediately follows from the one-to-one correspondence established in Theorem \ref{thm:embedding}, which means that for every $z\in f(\Acal)$ we find exactly one $y\in\Acal$ such that $f(y)=z$. Choose an arbitrary $y_0\in\Acal$ with $y_1=\Phi^\tau(y_0)\in \Acal$ (we have $y_1\in\Acal$ due to the invariance of $\Acal$). Then
\begin{equation*}
    \begin{aligned}
        h((\Kcal^\tau f)(y_0)) &= h(f(\Phi^\tau(y_0))) = h(f(y_1)) = h(z_1) \\
        &= h(\varphi^\tau(z_0)) = (\Kcalhat^\tau h)(z_0) = (\Kcalhat^\tau h)(f(y_0)) \\
        &= (\Kcalhat^\tau h \circ f)(y_0).
    \end{aligned}
\end{equation*}
For the spectrum, consider an eigenfunction $\hat\xi$ with associated eigenvalue $\hat\lambda$ such that
\begin{align*}
    \Kcalhat^\tau \hat\xi = \hat\lambda \hat\xi.
\end{align*}
Then, using Eq.\ \eqref{eq:CDS_PDE} and introducing $\xi = \hat\xi\circ f$, we find
\begin{align*}
    \hat\xi \circ \varphi^\tau = \hat\xi \circ f\circ\Phi^\tau\circ f^{-1}&= \hat\lambda \hat\xi \\
    \Leftrightarrow \qquad \hat\xi \circ f\circ\Phi^\tau &= \hat\lambda \hat\xi \circ f \\
    \Leftrightarrow \qquad \quad~~\xi\circ\Phi^\tau &= \hat\lambda \xi.
\end{align*}
\end{proof}
\begin{remark}
    For $h=\identity$, we find $\Kcal^\tau f = \Kcalhat^\tau h \circ f$.
\end{remark}
\begin{remark}
    Note that Theorem \ref{thm:Koopman_CDS} does not address the question of the approximation properties of the operators $\Kcal^\tau$ and $\Kcalhat^\tau$. The existence of a one-to-one relationship is very helpful in determining the required number of observables---provided that the attractor dimension can be estimated---but it is hard to assess whether the embedded dynamics are simpler or harder to approximate from a purely numerical perspective. This warrants further investigation, which is out of the scope of this paper.
\end{remark}

\begin{remark}[Coarse graining of ODEs]
	Even though we have only considered PDEs here, similar challenges occur in large-scale systems of ODEs such as molecular dynamics, agent-based systems or dynamics on graphs (e.g., electric grids, where the entire graph is not necessarily known). There, a coarse graining to meaningful macro observables \cite{ZHS16,BKK+17,BNC18,NKS21,NKBC21} is highly desirable, at the cost of losing knowledge of the underlying dynamical equations.
\end{remark}

\section{Combining equivariance and partial measurements}
\label{sec:EqivPlusPartial}
We will now demonstrate the simultaneous application of both above-mentioned concepts, meaning that we will compute Koopman operators based on point measurements from a small, connected subset of $\Omega$, which preserve the equivariance properties of the underlying PDE.
We will compare local Koopman models $\Kcalhat^\tau$---obtained from $q$ point measurements located in a compact subset of $[0,L]$ (e.g., neighboring grid points in the discretization)---to a global Koopman model $\Kcal^\tau$ which is obtained using the classical full state observable (i.e., we observe the entire numerical grid at once). As briefly mentioned above, these point measurements can be interpreted in terms of Eq.\ \eqref{eq:observable} by considering a Dirac delta function as the kernel. This way, we obtain 
\[ \begin{bmatrix} z_{t,1} = f_{x_1}(y_t)=y_t(x_1) & \ldots & z_{t,N} = y_t(x_N) \end{bmatrix}. \]

\begin{figure}[b]
	\centering
    \includegraphics[width=.7\columnwidth]{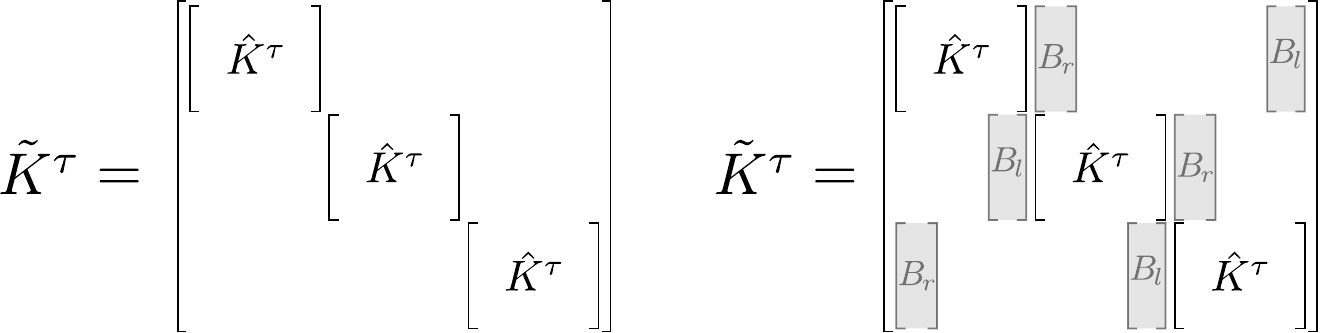}\\[.5em] \qquad(a)\hspace{4cm}(b)
	\caption{Schematic of the local Koopman approach. We consider a local Koopman matrix $\Khat^\tau \in \R^{q \times q}$. (a) The same approximation $\Khat^\tau$ can be applied anywhere in the domain such that we obtain a global matrix $\Ktilde^\tau$ with identical blocks $\Khat^\tau$. (b) The shaded $B$ terms represent coupling terms to neighboring local models if we pursue a DMDc-like approach.}
	\label{fig:Koopman_Conv}
\end{figure}

We will compare these two models both regarding the Koopman spectrum and the prediction accuracy. In the latter case, Fig.\ \ref{fig:Koopman_Conv} (a) illustrates how the local models $\Kcal^\tau$ can be combined to a global model of repeating entries. Conceptually, we simply obtain a number of entirely independent predictors for subsections of $\Ycal$. As we always have to deal with approximations of $\Kcalhat^\tau$, it is clear that the approximate solution can quickly lead to inconsistencies with respect to the PDE state, for instance by developing artificial discontinuities. The question is therefore whether we can establish a link between the different local models. A first intuitive approach would be to simply consider overlapping domains. For instance---considering the example of a local model of size three---we can apply the same $\Kcalhat^\tau$ to $\Psi((z_1,z_2,z_3))$ and to $\Psi((z_2,z_3,z_4))$.
However, this means that we effectively obtain multiple predictors for the same quantity (in the previous example, $z_2$ and $z_3$ are contained in both models). While it is certainly possible to project each of the local systems back onto the coordinates $z$ and then use the average as the predictor, we would still not achieve a coupling between different model instances.
\begin{remark}
    In fact, if we were able to obtain an exact finite-dimensional approximation of $\Kcalhat^\tau$, then we would quickly run into inconsistencies, as the same matrix would be a predictor for $(z_1,z_2,z_3)$ and for $(z_2,z_3,z_4)$. A quick calculation then shows that the prediction of any $z_i$ cannot depend on any other $z_j$, which would mean that $K^\tau$ is just a diagonal matrix. Our interpretation of this dilemma is that inconsistencies cannot be avoided, as in the generic situation, any finite-dimensional approximation is inexact, meaning that the local Koopman models have to be globally inconsistent or yield trivial dynamics.
    Note that this situation is different for other common observable functions such as Fourier modes. These have more sophisticated implications for the symmetry, as shown in Example \ref{ex:FourierObservable}
\end{remark}

Instead, a coupling can be achieved in two different ways. The first option is to treat the connection between two neighboring models as it is done in the Dynamic Mode Decomposition with control (DMDc, \cite{PBK15}), meaning that using a slightly modified regression problem, we obtain a model of the form
\begin{equation}\label{eq:DMDc}
    \begin{aligned}
    \Psi((z_{t+\tau,2},z_{t+\tau,3},z_{t+\tau,4})^\top) \approx \Khat^\tau \Psi((z_{t,2},z_{t,3},z_{t,4})^\top) &+ B_l z_{t,1} + B_r z_{t,5},
    \end{aligned}
\end{equation}
where the two terms $B_r, B_l\in\R^{q\times 1}$ are used to treat the left and right neighbors as control inputs to the dynamics.

As the above approach only works for control inputs that enter the original system in a purely linear fashion \cite{NPP+23}, an alternative approach is to accept that a purely linear approach may be too much to ask for. Instead, a coupling can be achieved by predicting the next state, project from $\Psi(z)$ onto the coordinates $z$, and then lift again for each system. Due to the project-then-lift step, we ``synchronize'' the local systems, at the cost of obtaining nonlinear dynamics. Similar observations in terms of trading the linearity for accuracy have been made in, e.g., \cite{CRLG23}. Note that if we use kernel EDMD with a radial basis function kernel, then this approach possesses a relation to Gaussian Process models.

\subsection{Numerical setup}
In the following, we will study the proposed approaches using the Kuramoto-Sivashinsky equation for two different parameter values $\mu>0$.
For the data generation process, we numerically solve \eqref{eq:KS} using the spectral Galerkin method implemented in the open source package \emph{shenfun} \cite{Mor18}. As a spatial discretization of $\Omega$, we use $N=32$ Fourier modes, which is equivalent to $N=32$ equidistant grid points, i.e., we have $\Delta x = \frac{L}{N} = \frac{\pi}{16}$. The time step for the PDE solver is $\Delta t=0.01$, and we set $\tau = 20\Delta t = 0.2$ for the traveling wave and $\tau = 5 \Delta t = 0.05$ for the bimodal fixed-point setting. We collect $M=1000$ samples from the attractor $\Acal$, which yields a sufficient coverage for the considered $\mu$ values.

In our experiments, we first compare the PDE solution to the approximation $K$ of the global Koopman operator $\Kcal$, where $f(y)=y$. We then compare this to local Koopman models $\Khat$ both in terms of the Koopman spectrum as well as regarding the prediction accuracy. 
For the latter, we construct a global model $\Ktilde$ from the local $\Khat$. We do so following both the classical approach (Fig.\ \ref{fig:Koopman_Conv} (a)) and the DMDc approach (Fig.\ \ref{fig:Koopman_Conv} (b)).
Following Theorem \ref{thm:embedding}, we study different embedding dimensions, where $q_w$ is the \emph{window width} (i.e., the number of neighboring points of the discretized PDE), $q_d$ is the number of delays, and the total dimension is $q=q_w\cdot q_d$. The global Koopman model $K$ is thus a special case of $\Khat$ where $q_w=N$. We will use standard DMD (i.e., $\Psi=\identity$) in all experiments.

\subsection{Traveling wave ($\mu=15$)}
At $\mu=15$, the system exhibits a traveling wave solution, as shown in Fig.\ \ref{fig:KS_mu15_PDE_vs_K}. As this is a very simple behavior, we do not study delay coordinates and set $q_d=1$, i.e., $q=q_w$. We see in Fig.\ \ref{fig:KS_mu15_PDE_vs_K} (bottom) that the DMD approximation (i.e., $\Psi=\identity$) is sufficient to yield accurate long-term predictions, even though a slow decay is visible after a while.
\begin{figure}[h]
    \centering
    \includegraphics[width=.7\columnwidth]{
        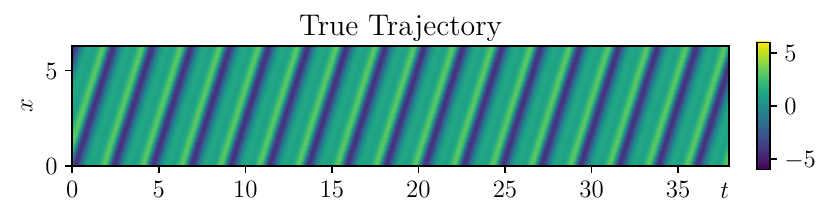}
    \includegraphics[width=.7\columnwidth]{
        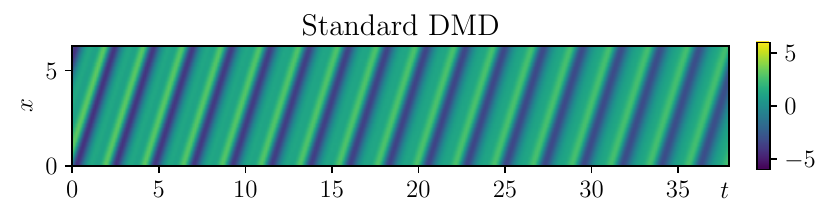}
    \caption{PDE solution vs.\ global Koopman-approximation for $\mu=15$. To compute $K$, we have used the standard DMD algorithm on the full state observable (i.e., $f=\Psi=\identity$). 
    }
    \label{fig:KS_mu15_PDE_vs_K}
\end{figure} 
\begin{figure}[h]
    \centering
    \includegraphics[width=0.4\columnwidth]{
        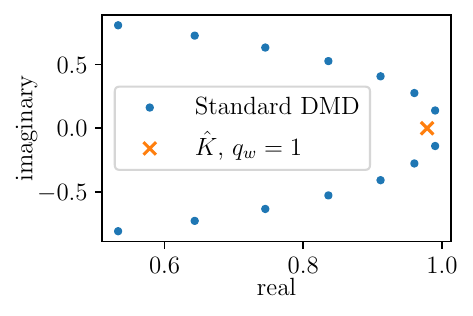}\hfil
    \includegraphics[width=0.4\columnwidth]{
        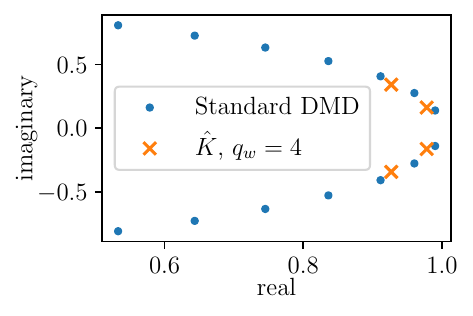}
    \includegraphics[width=0.4\columnwidth]{
        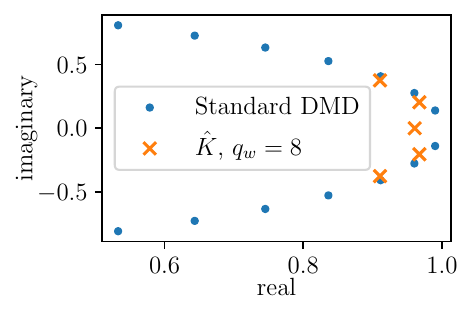} \hfil
    \includegraphics[width=0.4\columnwidth]{
        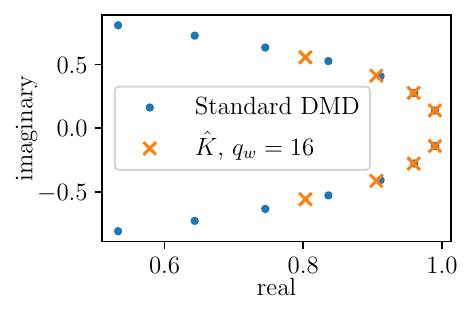}
    \caption{Eigenvalues of $K$ vs.\ $\Khat$ for varying $q$ values.}
    \label{fig:KS_mu15_spectra}
\end{figure}

We next compare $K$ and the local Koopman model $\Khat$ for different values of $q_w=q$, varying between $q=1$ and $q=16$, which is half of the domain. The corresponding spectra are compared in Fig.\ \ref{fig:KS_mu15_spectra}, and we observe a very good agreement for the leading eigenvalues (the lowest frequency corresponds to the frequency of the traveling wave).

The prediction of the state $y$ using the reconstructed Koopman operator $\Ktilde$ (according to Fig.\ \ref{fig:Koopman_Conv} (a)) is depicted in Fig.\ \ref{fig:KS_mu15_predictions}. As discussed before, the complete decoupling of the local models $\Khat$ yields globally inconsistent dynamics caused by small prediction errors. This is most evident for $q=4$ and $q=8$. At $q=16$, the approximation is sufficiently accurate that we do no longer observe this phenomenon. Interestingly, this decoherence is quite severe even though the one-step prediction error reaches a very small value already at $q=4$, cf.\ Fig.\ \ref{fig:KS_mu15_errors}.

\begin{figure}[h]
    \centering
    \includegraphics[width=.49\columnwidth]{
        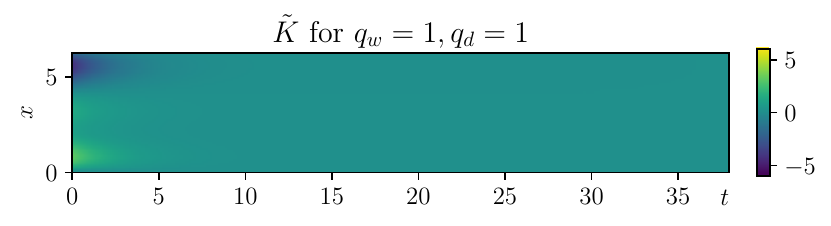}
    \includegraphics[width=.49\columnwidth]{
        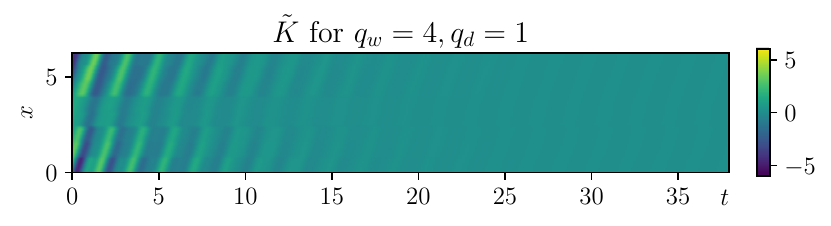}
    \includegraphics[width=.49\columnwidth]{
        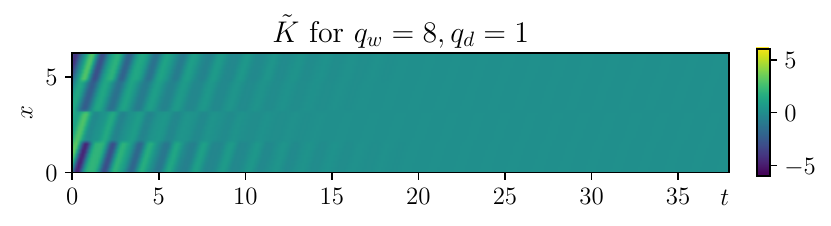}
    \includegraphics[width=.49\columnwidth]{
        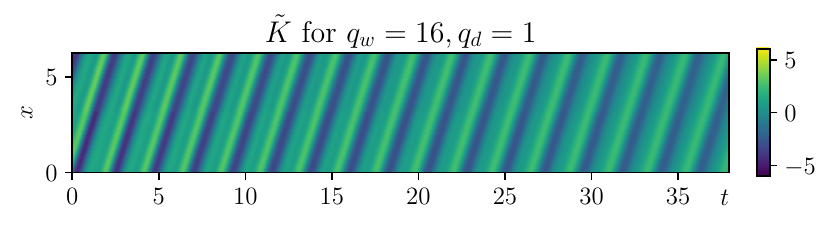}
    \caption{Predictions using $\Ktilde$ with varying $q$ values.}
    \label{fig:KS_mu15_predictions}
\end{figure}

\begin{figure}[h]
    \centering
    \includegraphics[width=0.4\textwidth]{
        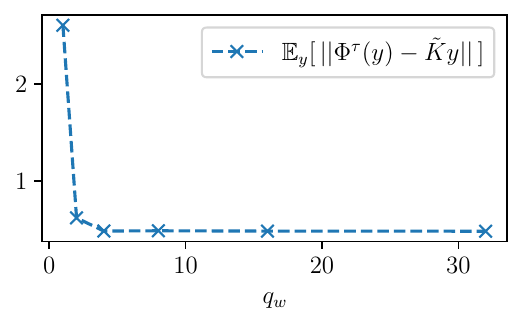}
    \caption{One-step prediction error for varying window widths $q_w$.}
    \label{fig:KS_mu15_errors}
\end{figure}

\begin{figure}[h]
    \centering
    \includegraphics[width=.6\columnwidth]{
        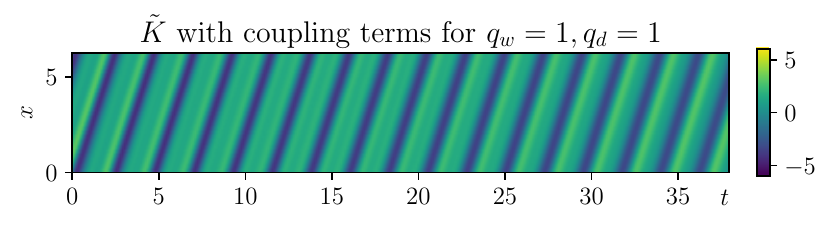}
    \caption{Local DMDc-approximation according to Fig.\ \ref{fig:Koopman_Conv} (b), with $q=1$ and an additional control input from left and right, respectively.}
    \label{fig:KS_mu15_DMDc}
\end{figure}

Unsurprisingly, the prediction accuracy improves massively when we build in a coupling term according to Fig.\ \ref{fig:Koopman_Conv} (b) and Eq.\ \eqref{eq:DMDc}. Using a single input from left and right, we obtain high-quality predictions using a very small linear system ($q=1$ and two inputs), see Fig.\ \ref{fig:KS_mu15_DMDc}.

\subsection{Bimodal fixed point ($\mu=18$)}
As a second system, we study the parameter $\mu=18$ which results in a ``checkerboard'' pattern. As the dynamics are now much more complex (due to the equivariance w.r.t.\ translations in space, the attractor consists of infinitely many checkerboard patterns), we here additionally consider $q_d=50$ delay observations. Using this higher-dimensional observable, many results from the previous case hold in a very similar fashion. Fig.\ \ref{fig:KS_mu18} shows a comparison of different approximations, more figures can be found in the appendix.
\begin{figure}[h]
    \centering
    \includegraphics[width=.6\columnwidth]{
        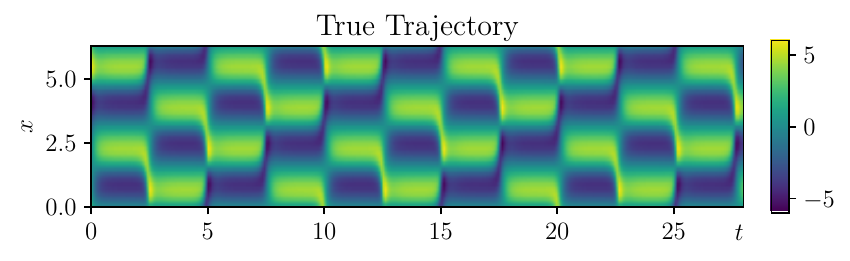}
    \includegraphics[width=.6\columnwidth]{
        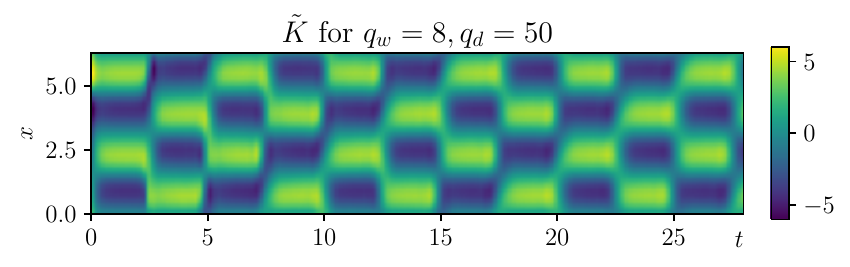}
    \includegraphics[width=.6\columnwidth]{
        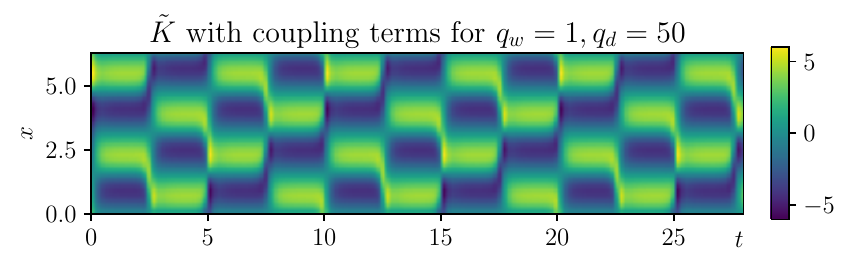}
    \caption{PDE vs.\ local model with $q_w=8$ and $q_d=50$ versus DMDc model with $q_w=1$ and $q_d=50$.}
    \label{fig:KS_mu18}
\end{figure}

\section{Conclusion}
We have presented two extensions to the current Koopman theory that deal with (i) the issue of not knowing the system's state space (i.e., using partial measurements, for instance from sensors) and (ii) the exploitation of symmetries when setting up Koopman-based surrogate models.
Regarding part (i), we have shown that there exists a close connection between the Koopman observable function $f$ and observation maps as they are defined in the embedding literature (Takens, Whitney, ...). If we observe sufficiently many points (more than two times the attractor dimension), then we can simply treat our measurements as if they have been generated by another dynamical system with state space $\R^q$, on which we can then use standard EDMD techniques. This yields rigorous criteria for the situation when building Koopman models exclusively from partial measurements.
For part (ii), we have then exploited this in order to simply use $q$ local measurements to build a local Koopman operator approximation.
Our numerical results show that including coupling terms (in the spirit of DMD with control) significantly increases the accuracy, even for very small model sizes.
The extension towards other symmetry groups as well as to two or three-dimensional spatial domains us currently under investigation, with a strong focus on equivariant convolution operators.

\section{Acknowledgments}

S.P.\ and H.H.\ acknowledge financial support by the project ``SAIL: SustAInable Life-cycle of Intelligent Socio-Technical Systems'' (Grant ID NW21-059D), which is funded by the program ``Netzwerke 2021'' of the Ministry of Culture and Science of the State of Northrhine Westphalia, Germany. K.W.\ gratefully acknowledges funding by the Deutsche Forschungsgemeinschaft (DFG, German Research Foundation) under the Project-ID 507037103.

\bibliographystyle{alpha}
\bibliography{references}

\newcommand{\etalchar}[1]{$^{#1}$}
\begin{thebibliography}{VDPWVH{\etalchar{+}}20}

\bibitem[AGS21]{AGS21}
K.~Atz, F.~Grisoni, and G.~Schneider.
\newblock Geometric deep learning on molecular representations.
\newblock {\em Nature Machine Intelligence}, 3:1023--1032, 2021.

\bibitem[AM17]{AM17}
H.~Arbabi and I.~Mezi\'{c}.
\newblock Ergodic theory, dynamic mode decomposition, and computation of
  spectral properties of the koopman operator.
\newblock {\em SIAM Journal on Applied Dynamical Systems}, 16(4):2096--2126,
  2017.

\bibitem[AM20]{AM20}
A.~M. Avila and I.~Mezi{\'{c}}.
\newblock {Data-driven analysis and forecasting of highway traffic dynamics}.
\newblock {\em Nature Communications}, 11:2090, 2020.

\bibitem[BBCV21]{BBCV21}
M.~M. Bronstein, J.~Bruna, T.~Cohen, and P.~Veli{\v{c}}kovi{\'c}.
\newblock {Geometric deep learning: Grids, groups, graphs, geodesics, and
  gauges}.
\newblock {\em arXiv:2104.13478}, 2021.

\bibitem[BBKK22]{BBKK22}
S.~L. Brunton, M.~Budi{\v{s}}i{\'{c}}, E.~Kaiser, and J.~N. Kutz.
\newblock {Modern Koopman Theory for Dynamical Systems}.
\newblock {\em {SIAM} Review}, 64(2):229--340, 2022.

\bibitem[BBL{\etalchar{+}}17]{BBL+17}
M.~M. Bronstein, J.~Bruna, Y.~LeCun, A.~Szlam, and P.~Vandergheynst.
\newblock {Geometric Deep Learning: Going beyond Euclidean data}.
\newblock {\em IEEE Signal Processing Magazine}, 34:18--42, 2017.

\bibitem[BBL{\etalchar{+}}23]{BML+23}
P.~Bevanda, M.~Beier, A.~Lederer, S.~Sosnowski, E.~H\"{u}llermeier, and
  S.~Hirche.
\newblock Koopman kernel regression.
\newblock In {\em Advances in Neural Information Processing Systems},
  volume~36, pages 16207--16221, 2023.

\bibitem[BBP{\etalchar{+}}17]{BBP+17}
S.~L. Brunton, B.~W. Brunton, J~L. Proctor, E.~Kaiser, and J.~N. Kutz.
\newblock Chaos as an intermittently forced linear system.
\newblock {\em Nature Communications}, 8, 2017.

\bibitem[BGSW24]{BGSW23}
L.~Bold, L.~Gr{\"u}ne, M.~Schaller, and K.~Worthmann.
\newblock Data-driven {MPC} with stability guarantees using extended dynamic
  mode decomposition.
\newblock {\em IEEE Transactions on Automatic Control}, 2024.

\bibitem[BKK{\etalchar{+}}17]{BKK+17}
A.~Bittracher, P.~Koltai, S.~Klus, R.~Banisch, M.~Dellnitz, and C.~Sch\"{u}tte.
\newblock {Transition Manifolds of Complex Metastable Systems: Theory and
  Data-Driven Computation of Effective Dynamics}.
\newblock {\em Journal of Nonlinear Science}, 28(2):471–512, 2017.

\bibitem[BMM12]{BMM12}
M.~Budi{\v{s}}i{\'{c}}, R.~Mohr, and I.~Mezi{\'{c}}.
\newblock {Applied Koopmanism}.
\newblock {\em Chaos}, 22, 2012.

\bibitem[BNC18]{BNC18}
L.~Boninsegna, F.~Nüske, and C.~Clementi.
\newblock Sparse learning of stochastic dynamical equations.
\newblock {\em Journal of Chemical Physics}, 148, 2018.

\bibitem[BPK16]{BPK16}
S.~L. Brunton, J.~L. Proctor, and J.~N. Kutz.
\newblock {Discovering governing equations from data by sparse identification
  of nonlinear dynamical systems}.
\newblock {\em Proceedings of the National Academy of Sciences},
  113(15):3932--3937, 2016.

\bibitem[CALG24]{CRLG23}
C.~R. Constante-Amores, A.~J. Linot, and M.~D. Graham.
\newblock {Enhancing Predictive Capabilities in Data-Driven Dynamical Modeling
  with Automatic Differentiation: Koopman and Neural ODE Approaches}, 2024.

\bibitem[CHK00]{CHK00}
A.~J. Chorin, O.~H. Hald, and R.~Kupferman.
\newblock {Optimal prediction and the Mori-Zwanzig representation of
  irreversible processes}.
\newblock {\em Proceedings of the National Academy of Sciences},
  97(7):2968--2973, 2000.

\bibitem[CKH22]{CKH22}
V.~Cibulka, M.~Korda, and T.~Haniš.
\newblock {Dictionary-free Koopman model predictive control with nonlinear
  input transformation}.
\newblock {\em arxiv:2212.13828}, 2022.

\bibitem[CW16]{CW16}
T.~S. Cohen and M.~Welling.
\newblock {Group Equivariant Convolutional Networks}.
\newblock In {\em International Conference on Machine Learning}, volume~33,
  pages 2990--2999, 2016.

\bibitem[DHZ16]{DHZ16}
M.~Dellnitz, M.~{Hessel-von Molo}, and A.~Ziessler.
\newblock {On the computation of attractors for delay differential equations}.
\newblock {\em Journal of Computational Dynamics}, 3(1):93--112, 2016.

\bibitem[HNZ86]{HNZ86}
J.~M. Hyman, B.~Nicolaenko, and S.~Zaleski.
\newblock {Order and complexity in the Kuramoto-Sivashinsky model of weakly
  turbulent interfaces}.
\newblock {\em Physica D: Nonlinear Phenomena}, 23(1):265--292, 1986.

\bibitem[HPN{\etalchar{+}}24]{HPN+24}
H.~Harder, S.~Peitz, F.~Nüske, F.~M. Philipp, M.~Schaller, and K.~Worthmann.
\newblock {Group Convolutional Extended Dynamic Mode Decomposition}.
\newblock {\em arXiv:2411.00905}, 2024.

\bibitem[HRV{\etalchar{+}}24]{HRV+24}
H.~Harder, J.~Rabault, R.~Vinuesa, M.~Mortensen, and S.~Peitz.
\newblock {Solving Partial Differential Equations with Equivariant Extreme
  Learning Machines}.
\newblock In {\em Symposium on Systems Theory in Data and Optimization (SysDO)
  (accepted; preprint: arXiv:2404.18530)}, 2024.

\bibitem[KKBK20]{KKBK20}
M.~Kamb, E.~Kaiser, S.~L. Brunton, and J.~N. Kutz.
\newblock {Time-Delay Observables for Koopman: Theory and Applications}.
\newblock {\em SIAM Journal on Applied Dynamical Systems}, 19(2):886–917,
  2020.

\bibitem[KKS16]{KKS16}
S.~Klus, P.~Koltai, and C.~Sch{\"{u}}tte.
\newblock {On the numerical approximation of the Perron-Frobenius and Koopman
  operator}.
\newblock {\em Journal of Computational Dynamics}, 3(1):51--79, 2016.

\bibitem[KM18a]{KM18}
M.~Korda and I.~Mezi{\'{c}}.
\newblock {Linear predictors for nonlinear dynamical systems: Koopman operator
  meets model predictive control}.
\newblock {\em Automatica}, 93:149--160, 2018.

\bibitem[KM18b]{KM18b}
M.~Korda and I.~Mezi{\'{c}}.
\newblock {On Convergence of Extended Dynamic Mode Decomposition to the Koopman
  Operator}.
\newblock {\em Journal of Nonlinear Science}, 28(2):687--710, 2018.

\bibitem[KNP{\etalchar{+}}20]{KNP+20}
S.~Klus, F.~Nüske, S.~Peitz, J.-H. Niemann, C.~Clementi, and C.~Schütte.
\newblock {Data-driven approximation of the Koopman generator: Model reduction,
  system identification, and control}.
\newblock {\em Physica D: Nonlinear Phenomena}, 406:132416, 2020.

\bibitem[Koo31]{Koo31}
B.~O. Koopman.
\newblock {Hamiltonian systems and transformation in Hilbert space}.
\newblock {\em Proceedings of the National Academy of Sciences},
  17(5):315--318, 1931.

\bibitem[KPB16]{KPB16}
J.~N. Kutz, J.~L. Proctor, and S.~L. Brunton.
\newblock {Koopman Theory for Partial Differential Equations}, 2016.

\bibitem[KPS{\etalchar{+}}24]{KPS+24}
F.~Köhne, F.~M. Philipp, M.~Schaller, A.~Schiela, and K.~Worthmann.
\newblock ${L}^\infty$-error bounds for approximations of the {K}oopman
  operator by kernel extended dynamic mode decomposition.
\newblock {\em SIAM Journal of Applied Dynamical Systems}, 2024.
\newblock accepted for publication, arXiv preprint arXiv.2403.18809.

\bibitem[Mau21]{Mau21}
A.~Mauroy.
\newblock {Koopman Operator Framework for Spectral Analysis and Identification
  of Infinite-Dimensional Systems}.
\newblock {\em Mathematics}, 9(19), 2021.

\bibitem[MBBV15]{MBB+15}
J.~Masci, D.~Boscaini, M.~M. Bronstein, and P.~Vandergheynst.
\newblock {Geodesic Convolutional Neural Networks on Riemannian Manifolds}.
\newblock In {\em IEEE international conference on computer vision workshops},
  pages 37--45, 2015.

\bibitem[Mez13]{Mez13}
I.~Mezi{\'{c}}.
\newblock {Analysis of Fluid Flows via Spectral Properties of the Koopman
  Operator}.
\newblock {\em Annual Review of Fluid Mechanics}, 45:357--378, 2013.

\bibitem[Mez22]{Mez22}
I.~Mezić.
\newblock {On Numerical Approximations of the Koopman Operator}.
\newblock {\em Mathematics}, 10(7):1180, 2022.

\bibitem[Mor18]{Mor18}
M.~Mortensen.
\newblock {Shenfun: High performance spectral Galerkin computing platform}.
\newblock {\em Journal of Open Source Software}, 3(31):1071, 2018.

\bibitem[NKBC21]{NKBC21}
F.~N\"{u}ske, P.~Koltai, L.~Boninsegna, and C.~Clementi.
\newblock {Spectral Properties of Effective Dynamics from Conditional
  Expectations}.
\newblock {\em Entropy}, 23(2):134, 2021.

\bibitem[NKS21]{NKS21}
J.-H. Niemann, S.~Klus, and C.~Sch\"{u}tte.
\newblock {Data-driven model reduction of agent-based systems using the Koopman
  generator}.
\newblock {\em {PLOS} {ONE}}, 16(5):e0250970, 2021.

\bibitem[NM20]{NM20}
H.~Nakao and I.~Mezić.
\newblock {Spectral analysis of the Koopman operator for partial differential
  equations}.
\newblock {\em Chaos: An Interdisciplinary Journal of Nonlinear Science},
  30(11), 2020.
\newblock 113131.

\bibitem[NPP{\etalchar{+}}23]{NPP+23}
F.~N{\"{u}}ske, S.~Peitz, F.~Philipp, M.~Schaller, and K.~Worthmann.
\newblock {Finite-data error bounds for Koopman-based prediction and control}.
\newblock {\em Journal of Nonlinear Science}, 33:14, 2023.

\bibitem[OBP21]{OP21}
S.~Ober-Bl{\"{o}}baum and S.~Peitz.
\newblock {Explicit multiobjective model predictive control for nonlinear
  systems with symmetries}.
\newblock {\em International Journal of Robust and Nonlinear Control},
  31(2):380--403, 2021.

\bibitem[OPR24]{OPR24}
S.~E. Otto, S.~Peitz, and C.~W. Rowley.
\newblock {Learning Bilinear Models of Actuated Koopman Generators from
  Partially-Observed Trajectories}.
\newblock {\em SIAM Journal on Applied Dynamical Systems}, 23(1):885--923,
  2024.

\bibitem[OR19]{OR19}
S.~E. Otto and C.~W. Rowley.
\newblock {Linearly Recurrent Autoencoder Networks for Learning Dynamics}.
\newblock {\em SIAM Journal on Applied Dynamical Systems}, 18(1):558--593,
  2019.

\bibitem[PBK15]{PBK15}
J.~L. Proctor, S.~L. Brunton, and J.~N. Kutz.
\newblock {Dynamic mode decomposition with control}.
\newblock {\em SIAM Journal on Applied Dynamical Systems}, 15(1):142--161,
  2015.

\bibitem[PHG{\etalchar{+}}18]{PHG+18}
J.~Pathak, B.~Hunt, M.~Girvan, Z.~Lu, and E.~Ott.
\newblock {Model-Free Prediction of Large Spatiotemporally Chaotic Systems from
  Data: A Reservoir Computing Approach}.
\newblock {\em Physical Review Letters}, 120(2), 2018.
\newblock 024102.

\bibitem[PK18]{PK18}
J.~Page and R.~R. Kerswell.
\newblock {Koopman analysis of Burgers equation}.
\newblock {\em Phys. Rev. Fluids}, 3:071901, 2018.

\bibitem[PK19]{PK19}
S.~Peitz and S.~Klus.
\newblock {Koopman operator-based model reduction for switched-system control
  of PDEs}.
\newblock {\em Automatica}, 106:184--191, 2019.

\bibitem[POR20]{POR20}
S.~Peitz, S.~E. Otto, and C.~W. Rowley.
\newblock {Data-Driven Model Predictive Control using Interpolated Koopman
  Generators}.
\newblock {\em SIAM Journal on Applied Dynamical Systems}, 19(3):2162--2193,
  2020.

\bibitem[PSB{\etalchar{+}}24]{PSB+24}
F.~M. Philipp, M.~Schaller, S.~Boshoff, S.~Peitz, F.~N{\"u}ske, and
  K.~Worthmann.
\newblock Variance representations and convergence rates for data-driven
  approximations of {K}oopman operators.
\newblock {\em arXiv:2402.02494}, 2024.

\bibitem[PSC{\etalchar{+}}24]{PSC+24}
S.~Peitz, J.~Stenner, V.~Chidananda, O.~Wallscheid, S.~L. Brunton, and
  K.~Taira.
\newblock {Distributed Control of Partial Differential Equations Using
  Convolutional Reinforcement Learning}.
\newblock {\em Physica D: Nonlinear Phenomena}, 461:134096, 2024.

\bibitem[PSW{\etalchar{+}}24]{PSW+24}
F.~Philipp, M.~Schaller, K.~Worthmann, S.~Peitz, and F.~N{\"{u}}ske.
\newblock {Error bounds for kernel-based approximations of the Koopman
  operator}.
\newblock {\em Applied and Computational Harmonic Analysis}, 71:101657, 2024.

\bibitem[RMB{\etalchar{+}}09]{RMB+09}
C.~W. Rowley, I.~Mezi{\'{c}}, S.~Bagheri, P.~Schlatter, and D.~S. Henningson.
\newblock {Spectral analysis of nonlinear flows}.
\newblock {\em Journal of Fluid Mechanics}, 641:115--127, 2009.

\bibitem[Rob05]{Rob05}
J.C. Robinson.
\newblock A topological delay embedding theorem for infinite-dimensional
  dynamical systems.
\newblock {\em Nonlinearity}, 18(5):2135--2143, 2005.

\bibitem[RPK19]{RPK19}
M.~Raissi, P.~Perdikaris, and G.~E. Karniadakis.
\newblock Physics-informed neural networks: {A} deep learning framework for
  solving forward and inverse problems involving nonlinear partial differential
  equations.
\newblock {\em Journal of Computational Physics}, 378:686--707, 2019.

\bibitem[Sch10]{Sch10}
P.~J. Schmid.
\newblock Dynamic mode decomposition of numerical and experimental data.
\newblock {\em Journal of Fluid Mechanics}, 656:5–28, 2010.

\bibitem[SER{\etalchar{+}}19]{SER+19}
A.~Salova, J.~Emenheiser, A.~Rupe, J.~P. Crutchfield, and R.~M. D'Souza.
\newblock Koopman operator and its approximations for systems with symmetries.
\newblock {\em Chaos}, 29:093128, September 2019.

\bibitem[SNY20]{sinha2020koopman}
S.~Sinha, S.P. Nandanoori, and E.~Yeung.
\newblock Koopman operator methods for global phase space exploration of
  equivariant dynamical systems.
\newblock {\em IFAC-PapersOnLine}, 53(2):1150--1155, 2020.

\bibitem[SWP{\etalchar{+}}23]{SWP+23}
M.~Schaller, K.~Worthmann, F.~Philipp, S.~Peitz, and F.~N{\"{u}}ske.
\newblock {Towards reliable data-based optimal and predictive control using
  extended DMD}.
\newblock {\em IFAC-PapersOnLine}, 56:169--174, 2023.

\bibitem[SYC91]{SYC91}
T.~Sauer, J.~A. Yorke, and M.~Casdagli.
\newblock Embedology.
\newblock {\em Journal of Statistical Physics}, 65(3–4):579–616, 1991.

\bibitem[Tak81]{Tak81}
F.~Takens.
\newblock Detecting strange attractors in turbulence.
\newblock In {\em Lecture Notes in Mathematics}, pages 366--381. Springer
  Berlin Heidelberg, 1981.

\bibitem[VBW{\etalchar{+}}18]{VBW+18}
P.~R. Vlachas, W.~Byeon, Z.~Y. Wan, T.~P. Sapsis, and P.~Koumoutsakos.
\newblock {Data-driven forecasting of high-dimensional chaotic systems with
  long short-Term memory networks}.
\newblock {\em Proceedings of the Royal Society A: Mathematical, Physical and
  Engineering Sciences}, 474(2213), 2018.

\bibitem[VDPWVH{\etalchar{+}}20]{PWH+20}
E.~Van Der~Pol, D.~E. Worrall, H.~Van~Hoof, F.~A. Oliehoek, and M.~Welling.
\newblock {MDP Homomorphic Networks: Group Symmetries in Reinforcement
  Learning}.
\newblock {\em Advances in Neural Information Processing Systems},
  33:4199--4210, 2020.

\bibitem[VRV{\etalchar{+}}23]{VRV+23}
C.~Vignon, J.~Rabault, J.~Vasanth, F.~Alc{\'{a}}ntara-{\'{A}}vila,
  M.~Mortensen, and R.~Vinuesa.
\newblock {Effective control of two-dimensional Rayleigh-B{\'{e}}nard
  convection: Invariant multi-agent reinforcement learning is all you need}.
\newblock {\em Physics of Fluids}, 35(6), 2023.

\bibitem[WAGK22]{WSGK22}
M.~Weissenbacher, S.~Abbott, A.~Garg, and Y.~Kawahara.
\newblock Koopman {Q}-learning: {Offline} {Reinforcement} {Learning} via
  {Symmetries} of {Dynamics}.
\newblock In {\em 39th {International} {Conference} on {Machine} {Learning}},
  pages 23645--23667, 2022.

\bibitem[Whi36]{Whi36}
H.~Whitney.
\newblock Differentiable manifolds.
\newblock {\em The Annals of Mathematics}, 37(3):645, 1936.

\bibitem[WKR15]{WKR15}
M.~O. Williams, I.~G. Kevrekidis, and C.~W. Rowley.
\newblock {A data--driven approximation of the Koopman operator: Extending
  dynamic mode decomposition}.
\newblock {\em Journal of Nonlinear Science}, 25:1307--1346, 2015.

\bibitem[WRK15]{WRK15}
M.~O. Williams, C.~W. Rowley, and I.~G. Kevrekidis.
\newblock {A kernel-based method for data-driven Koopman spectral analysis}.
\newblock {\em Journal of Computational Dynamics}, 2(2):247–265, 2015.

\bibitem[ZDG19]{ZDG19}
A.~Ziessler, M.~Dellnitz, and R.~Gerlach.
\newblock {The Numerical Computation of Unstable Manifolds for Infinite
  Dimensional Dynamical Systems by Embedding Techniques}.
\newblock {\em SIAM Journal on Applied Dynamical Systems}, 18(3):1265--1292,
  2019.

\bibitem[ZHS16]{ZHS16}
W.~Zhang, C.~Hartmann, and C.~Schütte.
\newblock Effective dynamics along given reaction coordinates, and reaction
  rate theory.
\newblock {\em Faraday Discussions}, 195:365--394, 2016.

\bibitem[ZZ23]{ZZ22}
C.~Zhang and E.~Zuazua.
\newblock {A quantitative analysis of Koopman operator methods for system
  identification and predictions}.
\newblock {\em Comptes Rendus. M{\'{e}}canique}, 351(S1):1--31, 2023.

\end{thebibliography}

\appendix
\section{Appendix}
\subsection{Additional plots for $\mu=18$}
We here show the same figures as for the traveling wave solution at $\mu=15$. The key difference in the numerical approximation is that we now consider delay observables, i.e., $f$ consists of $q_d=50$ delays and varying numbers $q_w$ of observed grid nodes.
\begin{figure}[h]
    \centering
    \includegraphics[width=.6\columnwidth]{
        graphics/mu18_bwd20_fwd20_delays50/true.pdf}
    \includegraphics[width=.6\columnwidth]{
        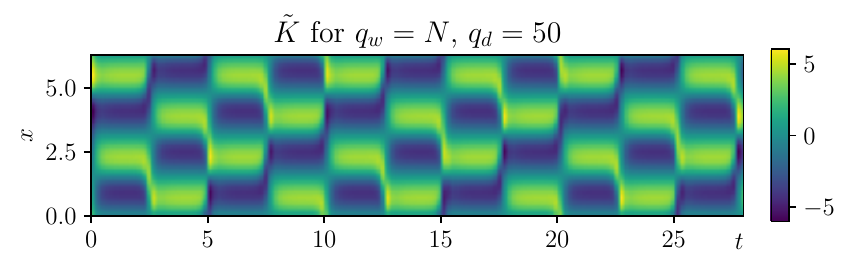}
    \caption{Numerical PDE solution versus global Koopman-approximation for $\mu=18$.}
    \label{fig:KS_mu18_PDE_vs_K}
\end{figure} 
\begin{figure}[h]
    \centering
    \includegraphics[width=0.4\columnwidth]{
        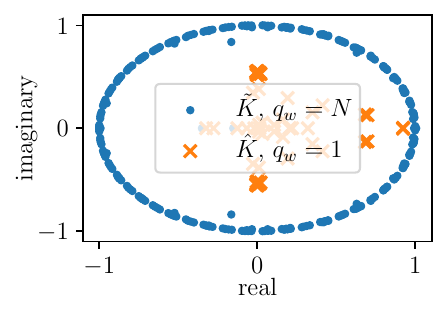}\hfil
    \includegraphics[width=0.4\columnwidth]{
        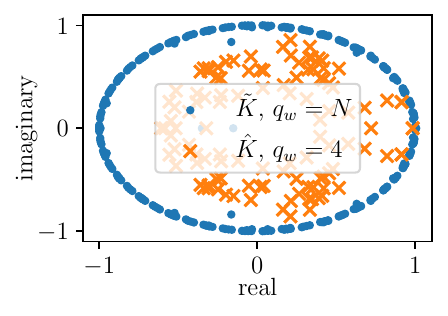}
    \includegraphics[width=0.4\columnwidth]{
        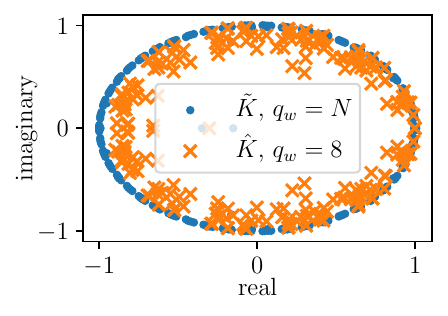}\hfil
    \includegraphics[width=0.4\columnwidth]{
        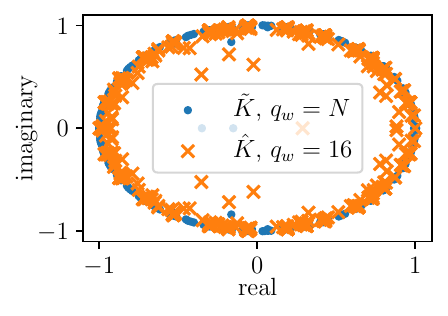}
    \caption{Eigenvalues of $K$ vs.\ $\Khat$ for varying $q$ values.}
    \label{fig:KS_mu18_spectra}
\end{figure}
\begin{figure}[h]
    \centering
    \includegraphics[width=.49\columnwidth]{
        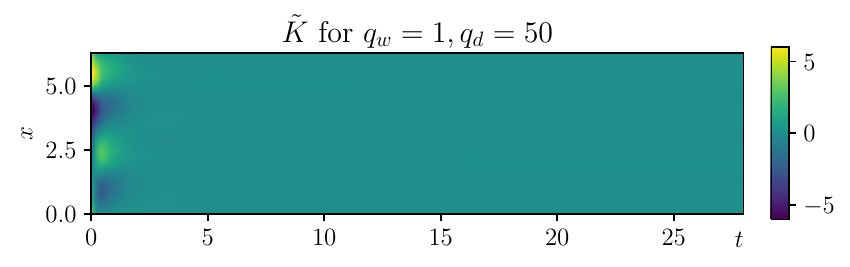}
    \includegraphics[width=.49\columnwidth]{
        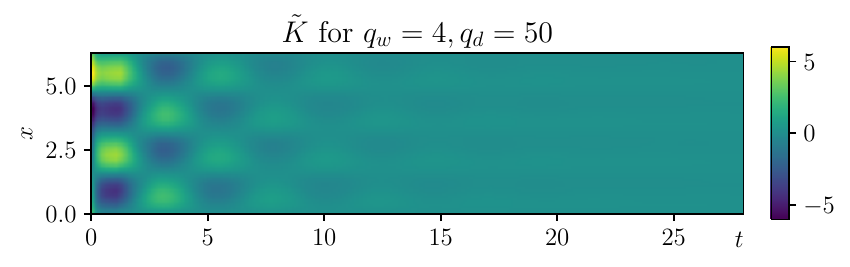}
    \includegraphics[width=.49\columnwidth]{
        graphics/mu18_bwd20_fwd20_delays50/K_hat_q8.pdf}
    \includegraphics[width=.49\columnwidth]{
        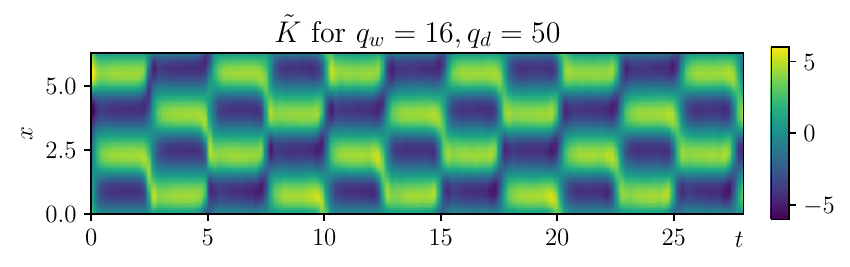}
    \caption{Predictions using $\Ktilde$ with varying $q$ values.}
    \label{fig:KS_mu18_predictions}
\end{figure}
\begin{figure}[h]
    \centering
    \includegraphics[width=0.5\textwidth]{
        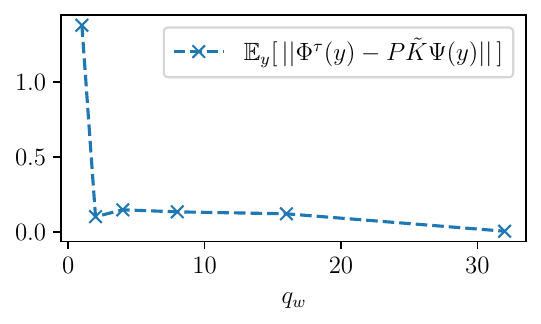}
    \caption{One-step prediction error of $\Ktilde$-predictions for varying $q$ values.}
    \label{fig:KS_mu18_errors}
\end{figure}
\begin{figure}[h]
    \centering
    \includegraphics[width=.6\columnwidth]{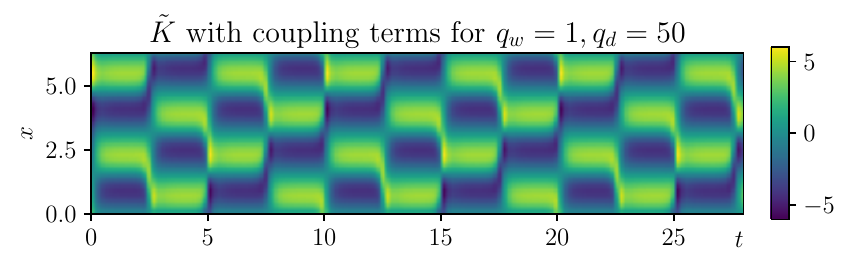}
    \caption{Local DMDc-approximation according to Fig.\ \ref{fig:Koopman_Conv} (b), with $q_d=50$, $q_w=1$ and an additional control input from left and right (all $q_d=50$ delays), respectively.}
    \label{fig:KS_mu18_DMDc}
\end{figure}

\end{document}